\documentclass[11pt]{amsart}
\usepackage[english]{babel}
\usepackage[utf8]{inputenc}
\usepackage[inner=1.1in, outer=1.1in, bottom =2.9cm, top =2.9cm]{geometry}
\makeatletter
\newcommand*\Cdot{\mathpalette\Cdot@{.5}}
\newcommand*\Cdot@[2]{\mathbin{\vcenter{\hbox{\scalebox{#2}{$\m@th#1\bullet$}}}}}
\makeatother

\usepackage{amsmath, amssymb, amsfonts, amsthm}
\usepackage{graphicx, overpic}
\usepackage{mathrsfs}
\usepackage{mathtools}
\usepackage[usenames,dvipsnames,svgnames,table]{xcolor}
\usepackage{enumerate}
\usepackage{enumitem}
\usepackage{shuffle} %Shuffle
\usepackage{lipsum}
\usepackage{color}
\usepackage{phoenician}
\usepackage{xkeyval}
\usepackage{mathrsfs}
\usepackage{tikz}
\usepackage{pst-node} % Commutative Diagrams
\usepackage{tikz-cd} % Commutative Diagrams
\usepackage[OT2,T1]{fontenc} %Shuffle
\usepackage{thm-restate} 
\usepackage{float}
\usepackage{eucal}
\usepackage{microtype}
\usepackage[strict]{changepage} %Margins
\usepackage{esvect}
\usepackage{etoolbox}
\usepackage{wrapfig}
\usepackage{caption}
\usepackage{cite}
\usepackage{mathrsfs}
\DeclareMathAlphabet{\mathpzc}{OT1}{pzc}{m}{it}
\usepackage{bbding}
\usepackage{array}
\usepackage{makecell}
\usepackage{stmaryrd}
\usepackage{lmodern}
\usepackage{amsbsy}
\usepackage{esvect}
\usepackage{bbding}
\usepackage{thm-restate}
\usepackage[toc,page, titletoc,title]{appendix}
\graphicspath{ {figures/} }
\usepackage{url}
\usepackage[pagebackref]{hyperref} %PUT LAST
\makeatletter
\providecommand*{\twoheadrightarrowfill@}{%
  \arrowfill@\relbar\relbar\twoheadrightarrow
}
\providecommand*{\twoheadleftarrowfill@}{%
  \arrowfill@\twoheadleftarrow\relbar\relbar
}
\providecommand*{\xtwoheadrightarrow}[2][]{%
  \ext@arrow 0579\twoheadrightarrowfill@{#1}{#2}%
}
\providecommand*{\xtwoheadleftarrow}[2][]{%
  \ext@arrow 5097\twoheadleftarrowfill@{#1}{#2}%
}
\makeatother
\makeatletter
\newcommand*{\relrelbarsep}{.386ex}
\newcommand*{\relrelbar}{%
  \mathrel{%
    \mathpalette\@relrelbar\relrelbarsep
  }%
}
\newcommand*{\@relrelbar}[2]{%
  \raise#2\hbox to 0pt{$\m@th#1\relbar$\hss}%
  \lower#2\hbox{$\m@th#1\relbar$}%
}
\providecommand*{\rightrightarrowsfill@}{%
  \arrowfill@\relrelbar\relrelbar\rightrightarrows
}
\providecommand*{\leftleftarrowsfill@}{%
  \arrowfill@\leftleftarrows\relrelbar\relrelbar
}
\providecommand*{\xrightrightarrows}[2][]{%
  \ext@arrow 0359\rightrightarrowsfill@{#1}{#2}%
}
\providecommand*{\xleftleftarrows}[2][]{%
  \ext@arrow 3095\leftleftarrowsfill@{#1}{#2}%
}
\makeatother
\makeatletter
\newcommand{\colim@}[2]{%
  \vtop{\m@th\ialign{##\cr
    \hfil$#1\operator@font colim$\hfil\cr
    \noalign{\nointerlineskip\kern1.5\ex@}#2\cr
    \noalign{\nointerlineskip\kern-\ex@}\cr}}%
}
\newcommand{\colim}{%
  \mathop{\mathpalette\colim@{\rightarrowfill@\scriptscriptstyle}}\nmlimits@
}
\renewcommand{\varprojlim}{%
  \mathop{\mathpalette\varlim@{\leftarrowfill@\scriptscriptstyle}}\nmlimits@
}
\renewcommand{\varinjlim}{%
  \mathop{\mathpalette\varlim@{\rightarrowfill@\scriptscriptstyle}}\nmlimits@
}
\makeatother

\DeclareSymbolFont{cyrletters}{OT2}{wncyr}{m}{n}
\DeclareMathSymbol{\Sh}{\mathalpha}{cyrletters}{"58}

\makeatletter
\newcommand*\bigcdot{\mathpalette\bigcdot@{.5}}
\newcommand*\bigcdot@[2]{\mathbin{\vcenter{\hbox{\scalebox{#2}{$\m@th#1\bullet$}}}}}
\makeatother

\usetikzlibrary{tikzmark,decorations.pathreplacing}
\usetikzlibrary{shapes,shadows,arrows}
\usetikzlibrary{decorations.markings}
\usetikzlibrary{positioning,calc}
\tikzset{near start abs/.style={xshift=1cm}}
\DeclareSymbolFont{symbolsC}{U}{txsyc}{m}{n}
\DeclareMathSymbol{\Searrow}{\mathrel}{symbolsC}{117}

\hypersetup{
	colorlinks=true,
	linkcolor=black,
	filecolor=black,      
	urlcolor=blue,
	citecolor= 	red,
}
\DeclareSymbolFont{extraup}{U}{zavm}{m}{n}
\DeclareMathSymbol{\varheart}{\mathalpha}{extraup}{86}
\DeclareMathSymbol{\vardiamond}{\mathalpha}{extraup}{87}
\theoremstyle{definition}
\newtheorem{thm}{Theorem}[section]
\newtheorem{cor}{Corollary}[thm]
\newtheorem{lem}[thm]{Lemma}
\newtheorem{prop}[thm]{Proposition}

\theoremstyle{definition}
\newtheorem{definition}{Definition}[section]

\newcommand{\gd}{\delta}

\newcommand{\cC}{\CMcal{C}}

\newcommand{\cG}{\CMcal{G}}

\newcommand{\cB}{\CMcal{B}}

\newcommand{\CAT}{\operatorname{CAT}}

\newcommand{\la}{\langle}
\newcommand{\ra}{\rangle}
\newcommand{\wt}{\widetilde}

\definecolor{Red}{rgb}{0.8,0,0.2}
\newcommand{\GG}[1]{}

\makeatletter
\def\@footnotecolor{red}
\define@key{Hyp}{footnotecolor}{%
 \HyColor@HyperrefColor{#1}\@footnotecolor%
}
\def\@footnotemark{%
    \leavevmode
    \ifhmode\edef\@x@sf{\the\spacefactor}\nobreak\fi
    \stepcounter{Hfootnote}%
    \global\let\Hy@saved@currentHref\@currentHref
    \hyper@makecurrent{Hfootnote}%
    \global\let\Hy@footnote@currentHref\@currentHref
    \global\let\@currentHref\Hy@saved@currentHref
    \hyper@linkstart{footnote}{\Hy@footnote@currentHref}%
    \@makefnmark
    \hyper@linkend
    \ifhmode\spacefactor\@x@sf\fi
    \relax
  }%
\makeatother

\hypersetup{footnotecolor=blue}
\makeatletter
\patchcmd{\@startsection}
  {\@afterindenttrue}
  {\@afterindentfalse}
  {}{}
\makeatother
\title[Stacky Buildings]{Stacky Buildings}
\author{William Norledge}
\address[William Norledge]{Pennsylvania State University}
\email{wxn39@psu.edu}
\begin{document}

% % % % % % % % % % % % % % % % % % % % % % % % %
% Autoref names
\renewcommand{\chapterautorefname}{Chapter}
\renewcommand{\sectionautorefname}{Section}
\renewcommand{\subsectionautorefname}{Section}
% % % % % % % % % % % % % % % % % % % % % % % % % 

% % % % % % % % % % % % % % % % % % % % % % % % % % % % % 
% Autoref names
\renewcommand{\chapterautorefname}{Chapter}
\renewcommand{\sectionautorefname}{Section}
\renewcommand{\subsectionautorefname}{Section}
% % % % % % % % % % % % % % % % % % % % % % % % % % % % %  

\begin{abstract}
We introduce structures which model quotients of buildings by type-preserving group actions. These structures, which we call W-groupoids for W a Coxeter group, generalize Bruhat decompositions, chambers systems of type M, Tits amalgams, and buildings themselves. We define the fundamental group of a W-groupoid, and characterize buildings as connected simply connected W-groupoids. We give a brief outline of covering theory of W-groupoids, which produces buildings as universal covers equipped with an action of the fundamental group. The local-to-global theorem of Tits concerning spherical 3-resides allows for the construction of W-groupoids by amalgamating quotients of generalized polygons along groupoids. In this way, W-groupoids provide a powerful way to construct (lattices in) exotic, hyperbolic, and wild buildings. 
\end{abstract}

\maketitle

% % % % % % % % % % % % % % % % % % % % % % % % % % % % % 
\setcounter{tocdepth}{1} % what appears
\hypertarget{foo}{ }
\tableofcontents
%\setlength{\cftbeforesecskip}{1pt} % spacing
% % % % % % % % % % % % % % % % % % % % % % % % % % % % % 

%%%%%%%%%%%%%%%%%%%%%%%%%%%%%%%%%%%%%%%%%%%%%%%%%%%%%%%%%%%%%%%%%%%%%%%%
%%%%%%%%%%%%%%%%%%%%%%%%%%%%%%%%%%%%%%%%%%%%%%%%%%%%%%%%%%%%%%%%%%%%%%%%
\section*{Introduction} 
%%%%%%%%%%%%%%%%%%%%%%%%%%%%%%%%%%%%%%%%%%%%%%%%%%%%%%%%%%%%%%%%%%%%%%%%
%%%%%%%%%%%%%%%%%%%%%%%%%%%%%%%%%%%%%%%%%%%%%%%%%%%%%%%%%%%%%%%%%%%%%%%%

Conceptually, a building is a set $\Delta$ equipped with a metric 
\[
\delta:\Delta\times \Delta \to W
\] 
which takes values in a Coxeter group $W=(W,S)$, with the Bruhat order of $W$ being used for the triangle inequality. The elements of $\Delta$ are called chambers, isometric embeddings $W\hookrightarrow \Delta$ are called apartments, paths in $\Delta$ are called galleries, and `geodesics' in $\Delta$ are called galleries of reduced type. This modern $W$-metric approach to buildings is described in \cite[Chapter 5]{ab08}, together with the equivalent simplicial complex approach taken in the early work of Tits \cite{tits74buildings}, \cite[Chapter 4]{ab08}. %The introduction of $W$-metric spaces was motivated by the discovery of twin buildings \cite{tits1992twin}. 
Buildings are also naturally $\CAT(0)$ cell complexes via the Davis realization \cite{davis94buildingscat0}, \cite[Chapter 12]{ab08}. If $W$ is an irreducible affine Coxeter group, then the Davis realization is the simplicial complex of the building (up to duality).

By the quotient of a building, we mean a would-be structure naturally associated to the type-preserving action of a group on a building in the manner of stacks, e.g. orbifolds, graphs/complexes of groups, algebraic stacks. Moreover, one should be able to develop covering theory for quotients, which is an important tool in the study of group actions. By modeling certain quotients of buildings as chamber systems of type $M$, covering theory of buildings was developed by Tits \cite{tits81local}. See also \cite{kantorscabs}, \cite{ronanlectures}, \cite{ronan92mainideas}. However, this theory `stops' at $2$-residues, and so is restricted to groups which act freely on the set of $2$-residues of a building. On the other hand, quotient data in the form of a Bruhat decomposition is associated to groups which act transitively on chambers and transitively on $W$-spheres \cite[Chapter 5]{ab08}. More generally, Tits's amalgam method \cite{tits85standrews}, \cite{tits86} constructs quotients of buildings by groups which act transitively on chambers. In this paper, we introduce structures called $W$-groupoids which unify and generalize these existing quotient constructions; $W$-groupoids are the full stacky generalization of buildings. In \cite{nor2}, (presentations of) $W$-groupoids which are quotients of chamber-free actions are studied, and in \cite{nor3}, $W$-groupoids are used to construct and classify new examples of Singer lattices in buildings.

Alternatively, if one takes the Davis realization of a building, then the quotient of a building will naturally be a developable complex of groups (since complexes of groups are the stacky quotients of cell complexes, see \cite[Chapter~III.$\cC$]{bridsonhae}). Constructing (lattices in) buildings by taking universal covers of complexes of groups is a common technique in geometric group theory, e.g. \cite{bourdon97}, \cite{gaboriaupaulin}, \cite{futer2011surface}, \cite{essert2013geometric}, \cite{Norledge2017}. If one restricts to torsion-free lattices, then the corresponding complex of groups has trivial local groups by the Bruhat-Tits fixed point theorem, and so cell complexes are sufficient to model quotients, e.g. \cite{cartwright1993groups}, \cite{vdovina02}. Unlike these geometric constructions, $W$-groupoids make use of the ultimate combinatorial $W$-metric structure of buildings. Importantly, the construction of quotients by amalgamating quotients of generalized polygons (and the calculation of their fundamental groups) is then straightforward, see e.g. \cite{nor3}. 

%%%%%%%%%%%%%%%%%%%%%%%%%%%%%%%%%%%%%%%%%%%%%%%%%%%%%%%%%%%%%%%%%%%%%%%%
\subsection*{Acknowledgments.} 
%%%%%%%%%%%%%%%%%%%%%%%%%%%%%%%%%%%%%%%%%%%%%%%%%%%%%%%%%%%%%%%%%%%%%%%%

We thank Alina Vdovina for useful discussions. This research was partly supported by a grant from Templeton Religion Trust as part of the mathematical picture language project at Harvard University. We also thank Newcastle University for their support. 

%%%%%%%%%%%%%%%%%%%%%%%%%%%%%%%%%%%%%%%%%%%%%%%%%%%%%%%%%%%%%%%%%%%%%%%%
%%%%%%%%%%%%%%%%%%%%%%%%%%%%%%%%%%%%%%%%%%%%%%%%%%%%%%%%%%%%%%%%%%%%%%%%
\section{$W$-Groupoid Properties} \index{$W$-groupoid!properties}
%%%%%%%%%%%%%%%%%%%%%%%%%%%%%%%%%%%%%%%%%%%%%%%%%%%%%%%%%%%%%%%%%%%%%%%%
%%%%%%%%%%%%%%%%%%%%%%%%%%%%%%%%%%%%%%%%%%%%%%%%%%%%%%%%%%%%%%%%%%%%%%%%

Let $\cG=(\cG_0,\cG_1)$ be a small groupoid with vertices/chambers (objects) $\cG_0$ and edges (morphisms) $\cG_1$, let $W=(W,S)$ be a Coxeter group with generating set $S$, and let `$\leq$' denote the Bruhat order of $W$. For an edge $g\in \cG$, we denote the initial and terminal chambers of $g$ by $\iota(g)$ and $\tau(g)$ respectively. For edges $g,h\in \cG$ such that $\tau(g)=\iota(h)$, we write 
\[gh=g;h=h\circ g\] 
for their composition in $\cG$. We make the convention that by `a function' 
\[
\gd:\cG\to W
\] 
we mean any function of sets $\gd:\cG_1\to W$. We call $\gd(g)$ the \emph{$W$-length} of the edge $g$. We call $\cG$ \emph{simply connected} if all its local groups are trivial. Simply connected groupoids are equivalent to equivalence relations, and if $\cG$ is a connected simply connected groupoid, then a function $\gd:\cG\to W$ is equivalent to a function $\gd:\cG_0\times \cG_0\to W$ by identifying the unique edge from $C$ to $D$ with $(C,D)\in \cG_0\times \cG_0$. 

We now introduce a collection of metric space like properties which an arbitrary function $\gd:\cG\to W$ may satisfy.
\begin{enumerate}  [itemindent=0cm, leftmargin=1.9cm]
\item [\textbf{(WG1)}]
For all identity edges $1\in \cG$, we have  $\gd(1)=1$. 
\end{enumerate}
In particular, (WG1) allows a non-identity edge of $\cG$ to have a $W$-length of $1$. 
\begin{enumerate}  [itemindent=0cm, leftmargin=1.9cm] \index{$W$-groupoid!properties!triangle inequality}
\item [\textbf{(WG2)}]
For all edges $g,h\in \cG$ such that $gh$ is defined in $\cG$, we have
\[\delta(gh)\geq \delta(g)\delta(h).\]
\end{enumerate}
Schematically, we have
\begin{center}
	\begin{tikzpicture}
	
	\node (11) at (0,0) {};
	\node (21) at (-1,2) {};
	\node (22) at (1,2) {};

	\draw[fill=black] (11) circle [radius=0.06];
	\draw[fill=black] (21) circle [radius=0.06];
	\draw[fill=black] (22) circle [radius=0.06];

	\draw [black, ->, shorten >=0.2cm] (11.center) to (21.center);
	\draw [black, ->, shorten >=0.2cm] (11.center) to (22.center);
	\draw [black, ->, shorten >=0.2cm] (22.center) to (21.center);
	
	\node[xshift=-1.9cm] at (-0.5,1) {$\delta(gh)\geq \delta(g)\delta(h)$}; 
	\node[xshift=0.4cm] at (0.5,1) {$g$}; 
	\node[yshift=0.4cm] at (0,2) {$h$}; 
	
	\end{tikzpicture}
\end{center}
Property (WG2) will be called the \emph{triangle inequality}, for obvious reasons, although the direction of the inequality is opposite to what one might expect.
\begin{enumerate}  [itemindent=0cm, leftmargin=1.9cm] \index{$W$-groupoid!properties!local triangle inequality}
\item [\textbf{(WG2$'$)}]
For all edges $g,h\in \cG$ such that $g^{-1}h$ is defined in $\cG$ and $\gd(g^{-1}h)\in S$, putting $w=\gd(g)$ and $s=\gd(g^{-1}h)$, we have:
\begin{enumerate}[label=(\roman*)]
\item
if $ws<w$, then $\gd(h)\in \{w, ws\}$
\item
if $ws>w$, then $\gd(h)=ws$.
\end{enumerate}
\end{enumerate}
Schematically, we have
\begin{center}
	\begin{tikzpicture}
	\node (11) at (0,0) {};
	\node (21) at (-0.6,3) {};
	\node (22) at (0.6,3) {};
	
	\draw[fill=black] (11) circle [radius=0.06];
	\draw[fill=black] (21) circle [radius=0.06];
	\draw[fill=black] (22) circle [radius=0.06];
	
	\draw [black, ->, shorten >=0.2cm] (11.center) to (21.center);
	\draw [black, ->, shorten >=0.2cm] (11.center) to (22.center);
	\draw [black, ->, shorten >=0.2cm] (22.center) to (21.center);
	
	\node[xshift=-0.8cm] at (-0.5,1.5) {$\{w, ws\}$}; 
	\node[xshift=0.17cm] at (0.5,1.5) {$w$}; 
	\node[yshift=0.3cm] at (0,3) {$s$}; 
	\end{tikzpicture}
\end{center}
Property (WG2$'$) will be called the \emph{local triangle inequality}. Roughly speaking, the local triangle inequality is the triangle inequality where one of the sides of the triangle has $W$-length a generator $s\in S$. We will describe the precise relationship between the triangle inequality and its local version in \autoref{sec:gallerys and geodesics}.
\begin{enumerate}  [itemindent=0cm, leftmargin=1.9cm]
\item [\textbf{(WG3)}]
For all edges $g\in \cG$ and for each $s\in S$ such that $\gd(g)s<\gd(g)$, there exists an edge $h\in \cG$ with $\iota(h)=\iota(g)$ such that
\[\gd(h^{-1}g)=s\qquad   \text{and}\qquad   \gd(h)=\gd(g)s.\]
\end{enumerate}
Schematically, we have
\begin{center}
	\begin{tikzpicture}
	\node (11) at (0,0) {};
	\node (21) at (0,3) {};
	\node (22) at (-0.6,2.2) {};
	
	\draw[fill=black] (11) circle [radius=0.06];
	\draw[fill=black] (21) circle [radius=0.06];
	\draw[fill=black] (22) circle [radius=0.06];

	\draw [black, ->, shorten >=0.2cm] (11.center) to (21.center);
	\draw [black, ->, shorten >=0.2cm] (11.center) to (22.center);
	\draw [black, ->, shorten >=0.2cm] (22.center) to (21.center);
	
	\node[xshift=-1.35cm] at (-0.3,1.1) {$\gd(h)=\gd(g)s$}; 
	\node[xshift=0.3cm] at (0,1.5) {$g$}; 
	\node[yshift=0cm] at (-0.6,2.7) {$s$}; 
	\end{tikzpicture}
\end{center}
Property (WG3) can be viewed as the analog of being a geodesic metric space in classical metric geometry. We make this more precise in \autoref{sec:gallerys and geodesics}. Finally, we call a function $\gd:\cG\to W$ \emph{weak} if for all chambers $C\in \cG$ and all $s\in S$, there exists a edge $g\in \cG$ with $\iota(g)=C$ and $\gd(g)=s$. Property (WG3) together with weakness is analogous to the existence of geodesic rays in metric geometry.

%%%%%%%%%%%%%%%%%%%%%%%%%%%%%%%%%%%%%%%%%%%%%%%%%%%%%%%%%%%%%%%%%%%%%%%%
%%%%%%%%%%%%%%%%%%%%%%%%%%%%%%%%%%%%%%%%%%%%%%%%%%%%%%%%%%%%%%%%%%%%%%%%
\section{Some Consequences}
%%%%%%%%%%%%%%%%%%%%%%%%%%%%%%%%%%%%%%%%%%%%%%%%%%%%%%%%%%%%%%%%%%%%%%%%
%%%%%%%%%%%%%%%%%%%%%%%%%%%%%%%%%%%%%%%%%%%%%%%%%%%%%%%%%%%%%%%%%%%%%%%%

We now prove some easy consequences of properties (WG1), (WG2), and (WG2$'$).

\begin{lem} \label{lem:inverses2}
Let $\gd:\cG\to W$ be a function which satisfies properties (WG1) and (WG2$'$). For any edge $g\in \cG$ with $\gd(g)\in S$, we have $\gd(g^{-1})=\gd(g)$.
\end{lem}

\begin{proof}
Let $\gd(g)=s$, and $\gd(g^{-1})=w$. In the definition of (WG2$'$), put $h=g^{-1}g$, and redefine $g$ to be $g^{-1}$. Then we get
\[1=\gd(h)  \in \{w,ws\} .\] 
\noindent But by (WG1), $w\neq 1$. Thus, $1=ws$ and $\gd(g^{-1})=w=s$.
\end{proof}

If we replace (WG2$'$) with (WG2) in \autoref{lem:inverses2}, then we obtain the following stronger result.

\begin{lem} \label{lem:inverses}
Let $\gd:\cG\to W$ be a function which satisfies properties (WG1) and (WG2). For any edge $g\in \cG$, we have $\gd(g^{-1})=\gd(g)^{-1}$.
\end{lem}

\begin{proof}
Using (WG1) and (WG2), we get 
\[1=\gd(1)=\gd(gg^{-1})\geq \gd(g)\gd(g^{-1}).\] 
Therefore $\gd(g)\gd(g^{-1})=1$, and the result follows. 
\end{proof}    

\begin{lem} \label{prop:doublecoset}
Let $\gd:\cG\to W$ be a function which satisfies properties (WG1) and (WG2). For all edges $g,h\in \cG$ such that $g^{-1} h$ is defined in $\cG$ and $\gd(g^{-1} h)=1$, we have $\gd(g)=\gd(h)$.
\end{lem}

\begin{proof}
Using (WG2) and \autoref{lem:inverses}, we get
\[1=\gd(g^{-1} h)\geq \gd(g^{-1})\gd(h)=\gd(g)^{-1}\gd(h).\] 
\noindent So $\gd(g)^{-1}\gd(h)=1$, or equivalently $\gd(g)=\gd(h)$.
\end{proof} 

%This implies that every premetric groupoid has a canonical quotient which is metric, and any metric groupoid can be `fattened up' by replacing objects with several objects all at distance `$1$' from each other. 

\begin{lem} \label{lem:panelsclosed}
Let $\gd:\cG\to W$ be a function which satisfies properties (WG1) and (WG2). Let $g,h\in \cG$ be edges such that $gh$ is defined in $\cG$, and $\gd(g), \gd(h)\in S$. Then we have the following:

\begin{enumerate}[label=(\roman*)]
\item
if $\gd(g)=\gd(h)=s$, then $\gd(gh)\in \{1,s\}$ 
\item
if  $\gd(g)\neq\gd(h)$, then $\gd(gh)=\gd(g)\gd(h)$.
\end{enumerate}
\end{lem}

\begin{proof}
Suppose that $\gd(g)=\gd(h)=s$, and let $k=gh$. Then 
\[s =\gd(h)  =  \gd (g^{-1}k)\geq \gd (g^{-1})\gd (k)=s\gd(k).\]
\noindent Thus, $s\gd(k)\in \{1,s  \}$, and so $\gd(gh)=\gd(k) \in \{1,s\}$. Now suppose that $\gd(g)\neq\gd(h)$. Let $s=\gd(g)$ and $t=\gd(h)$. Then
\[t =\gd(h)  =  \gd (g^{-1}k)\geq \gd (g^{-1})\gd (k)=s\gd(k).\]
\noindent Thus, $s\gd(k)\in \{1,t  \}$, and so $\gd(k) \in \{s,st\}$. But
\[   \gd(k)=\gd(gh)\geq st  . \]
\noindent Therefore we must have $\gd(gh)=\gd(k)=st$.
\end{proof}

The following lemma shows that in presence of (WG1), the local triangle inequality is weaker than the triangle inequality.

\begin{lem} \label{lem:wg2implieswg2'}
Let $\gd:\cG\to W$ be a function which satisfies properties (WG1) and (WG2). Then $\gd$ satisfies the local triangle inequality (WG2$'$). 
\end{lem}

\begin{proof}
Let $g,h\in \cG$ with $\gd(g^{-1}h)=s$. Let $k=g^{-1}h$. Then
\[s=\delta(k)=\delta(g^{-1}h)       \geq \delta(g)^{-1}\delta(h). \]
\noindent Thus, $\delta(g)^{-1}\delta(h)\in \{1,s\}$, and so $\gd(h)\in \{\gd(g), \gd(g)s\}$. Also
\[\delta(h) = \delta(gk)\geq \delta(g)\delta(k)=\delta(g)s .       \]
\noindent Thus, if $\gd(g)s>\gd(g)$, then $\gd(h)=\gd(g)s$.
\end{proof}

%In the presence of (WG1) and (WG2), we get the symmetrical version of (WG3).
%\begin{prop}
%Let $\gd:\cG\to W$ be a function satisfying (WG1), (WG2), and (WG3). Then for all $g\in \cG$, for each $s\in S$ such that $s\gd(g)<\gd(g)$, there exists $h$ with $\gd(gh^{-1})=s$ and $\gd(h)=s\gd(g)$.
%\end{prop}
%\begin{proof}
%Easy consequence of \autoref{lem:inverses}.
%\end{proof}

%%%%%%%%%%%%%%%%%%%%%%%%%%%%%%%%%%%%%%%%%%%%%%%%%%%%%%%%%%%%%%%%%%%%%%%%
%%%%%%%%%%%%%%%%%%%%%%%%%%%%%%%%%%%%%%%%%%%%%%%%%%%%%%%%%%%%%%%%%%%%%%%%
\section{Galleries and Geodesics.} \label{sec:gallerys and geodesics} \index{$W$-groupoid!geodesic}
%%%%%%%%%%%%%%%%%%%%%%%%%%%%%%%%%%%%%%%%%%%%%%%%%%%%%%%%%%%%%%%%%%%%%%%%
%%%%%%%%%%%%%%%%%%%%%%%%%%%%%%%%%%%%%%%%%%%%%%%%%%%%%%%%%%%%%%%%%%%%%%%%

We now define galleries and geodesics (galleries of reduced type) for functions $\gd:\cG\to W$. We show that property (WG3) ensures that we have geodesics, and we prove the equivalence of the local and global triangle inequalities in the presence of (WG1) and (WG3). 

Let $\gd:\cG\to W$ be a function. A \emph{gallery} of an edge $g\in \cG$ is a word $g_1\dots g_n$, where $g_k\in \cG_1$, such that
\begin{enumerate} [label=(\roman*)]
\item
$\gd(g_k)\in S$ for all $k\in \{ 1,\dots,n  \}$
\item
$g_k g_{k+1}$ is defined in $\cG$ for all $k\in \{  1,\dots,n-1 \}$
\item
$g=g_1\dots g_n$.
\end{enumerate}
We call the word $\gd(g_1)\dots \gd(g_n)$ over $S$ the \emph{type} of the gallery $g_1\dots g_n$. A \emph{geodesic} of $g$ is a gallery $g=g_1\dots  g_n$ whose type is a reduced word for $(W,S)$. The following proposition shows that in the presence of (WG2$'$) and (WG3), geodesics behave as their name implies.

\begin{prop} \label{prop:reduceddecompgeos}
Let $\delta:\cG\to W$ be a function which satisfies properties (WG2$'$) and (WG3). Let $g=g_1 \dots g_n$ be a geodesic of a edge $g\in \cG$. Then 
\[\gd(g)=\gd(g_1)\dots \gd(g_n).\]
\end{prop}

\begin{proof}
We prove by induction on $n$. The result is trivial for $n=1$. Suppose that the result holds for $n-1$. Let $h=gg_n^{-1}$. Then, by the induction hypothesis, we have 
\[\gd(h)=\gd(g_1)\dots \gd(g_{n-1}).\] 
But $g=hg_n$ and $\gd(h)\gd(g_n)>\gd(h)$. Therefore $\gd(g)=\gd(g_1)\dots \gd(g_n)$ by the local triangle inequality.
\end{proof}

We now show that (WG3) can be viewed as the property that `all possible' geodesics exist.
 
\begin{lem} \label{geodesicmetricspace}
Let $\delta:\cG\to W$ be a function which satisfies property (WG3). Let $g\in \cG$ be a edge and put $w=\gd(g)$. For every reduced decomposition $f$ of $w$, there exists a geodesic of $g$ whose type is $f$.  
\end{lem}

\begin{proof}
We prove by induction on the length $n$ of $f$. The result is trivial for $n=1$, since in this case $g$ is a geodesic of itself. Suppose that the result holds for $n-1$. Let $f=s_1\dots s_n$. By (WG3), there exists a edge $h\in \cG$ such that $\gd(h)=ws_n$ and $\gd(h^{-1}g)=s_n$. Put $g_n=h^{-1}g$. Using the induction hypothesis, let $g_1 \dots g_{n-1}$ be a geodesic of $h$ with type $s_1\dots s_{n-1}$. Then $g_1 \dots g_{n-1}g_n$ is the required geodesic of $g$.
\end{proof}

In fact, for each $f$, the geodesic is unique. Using geodesics, we can now strength the result of \autoref{lem:inverses2} by adding (WG3) to the hypothesis.

\begin{cor} \label{cor:inv}
Let $\delta:\cG\to W$ be a function which satisfies properties (WG1), (WG2$'$), and (WG3). Then for any edge $g\in \cG$, we have $\gd(g^{-1})=\gd(g)^{-1}$.
\end{cor}

\begin{proof}
By \autoref{geodesicmetricspace}, there exists a geodesic $g_1\dots g_n$ of $g$. Then $g_n^{-1}\dots g_1^{-1}$ is a geodesic of $g^{-1}$. But $\gd(g_j^{-1})=\gd(g_j)$ by \autoref{lem:inverses2}. The result then follows by \autoref{prop:reduceddecompgeos}.
\end{proof}

The following proposition shows that in the presence of (WG1) and (WG3), the local triangle inequality implies the triangle inequality. 

\begin{prop}  \label{lem:localtri}
Let $\gd:\cG\to W$ be a function which satisfies properties (WG1), (WG2$'$), and (WG3). Then $\gd$ satisfies property (WG2).
\end{prop}

\begin{proof}
Let $g,h\in \cG$ be edges such that $gh$ is defined in $\cG$, and let 
\[gh=g_1\dots g_n\] 
be a geodesic of $gh$ with type $f$. Then 
\[\gd(h g_n^{-1}) \in \big\{\gd(h),\gd(h)\gd(g_n) \big \}\] 
by (WG2$'$). By proceeding inductively in $n$, we see that $\gd(g^{-1})=\gd(h)f'^{-1}$, where $f'$ is a substring of $f$. Then $\gd(g)=f' \gd(h)^{-1}$. Thus, $\gd(g)\gd(h)=f'$, and it follows that $\gd(gh)\geq \gd(g)\gd(h)$ in the Bruhat order.   
\end{proof}

By combining \autoref{lem:localtri} with \autoref{lem:wg2implieswg2'}, we see that in the presence of properties (WG1) and (WG3), the triangle inequality and the local triangle inequality are equivalent.

%%%%%%%%%%%%%%%%%%%%%%%%%%%%%%%%%%%%%%%%%%%%%%%%%%%%%%%%%%%%%%%%%%%%%%%%
%%%%%%%%%%%%%%%%%%%%%%%%%%%%%%%%%%%%%%%%%%%%%%%%%%%%%%%%%%%%%%%%%%%%%%%%
\section{Definition of $W$-Groupoids} \label{section:defofweylmetricgroupoid}
%%%%%%%%%%%%%%%%%%%%%%%%%%%%%%%%%%%%%%%%%%%%%%%%%%%%%%%%%%%%%%%%%%%%%%%%
%%%%%%%%%%%%%%%%%%%%%%%%%%%%%%%%%%%%%%%%%%%%%%%%%%%%%%%%%%%%%%%%%%%%%%%%

We now define $W$-groupoids.
\begin{definition} Let $W$ be a Coxeter group, and let $\cG$ be a groupoid. A \emph{$W$-groupoid} on $\cG$ is a weak function $\delta:\cG\to W$ which satisfies the following three properties:
\begin{enumerate}[itemindent=0cm, leftmargin=1.9cm]
\item [\textbf{(WG1)}]
For all identity edges $1\in \cG$, we have  $\gd(1)=1$ 
\item [\textbf{(WG2)}]
For all edges $g,h\in \cG$ such that $gh$ is defined in $\cG$, we have 
\[\delta(gh)\geq \delta(g)\delta(h)\]
\item [\textbf{(WG3)}]
For all edges $g\in \cG$ and for each $s\in S$ such that $\gd(g)s<\gd(g)$, there exists an edge $h\in \cG$ with $\iota(h)=\iota(g)$ such that
\[
\gd(h^{-1}g)=s\qquad   \text{and}\qquad \gd(h)=\gd(g)s
.\]
\end{enumerate} 
\end{definition}
\noindent If in addition $\gd(g)=1$ implies that $g$ is an identity edge, then $\delta:\cG\to W$ is called a \emph{strict $W$-groupoid}. We show in \autoref{buildings} that buildings are equivalently connected simply connected strict $W$-groupoids. It follows from weakness that $\delta$ is always surjective. By our previous results, $W$-groupoids satisfy the local triangle inequality, and the property that $\gd(g^{-1})=\gd(g)^{-1}$ for all edges $g\in \cG$. For a chamber $C\in \cG$, the \emph{fundamental group} $\pi_1(\delta,C)$ of $\delta$ at $C$ is the local group of $\cG$ at $C$. The \emph{fundamental groupoid} $\pi(\delta)$ of $\delta$ is just $\cG$.

For $J\subseteq S$, we denote by $\cG_J$ the restriction of $\cG$ to those edges whose $W$-length is an element of the parabolic subgroup $W_J=\la J \ra\leq W$. Then $\cG_J$ is naturally a $W_J$-groupoid. The \emph{Borel subgroupoid} $\cB$ of $\cG$ is $\cG_{\emptyset}$, i.e. it is the subgroupoid of $\cG$ whose set of edges is given by
\[   \cB_1:=\big \{g\in \cG_1:\gd(g)=1\big \}   .   \] 
Thus, a $W$-groupoid is strict if and only if its Borel subgroupoid is a bundle of trivial groups. For $J\subseteq S$, we call a connected component of $\cG_J$ a \emph{$J$-residue}. If $|J|=2$, we say \emph{$2$-residue}. If $J=\{s\}$ for a generator $s\in S$, then we call $\cG_J$ the \emph{panel groupoid} of type $s$. Notice that to determine the function $\delta$ of a $W$-groupoid $\cG$, it is sufficient to know the panel groupoids $\cG_s\leq \cG$ since geodesics will then tell us the value of $\delta$ on all other edges (in the same way that chamber systems determine buildings).    

\section{$W$-Groupoids and Buildings} \label{buildings} We now demonstrate the connection between $W$-groupoids and buildings. For our definition of a building, we take the symmetrical version of the axioms which appear in \cite[p.~218]{ab08}, which are equivalent by \cite[Remark 5.18]{ab08}. 

\begin{definition}
Let $W=(W,S)$ be a Coxeter group. A \emph{building} $(\Delta,\gd)$ of type $W$ is a set of \emph{chambers} $\Delta$ equipped with a function $\gd:\Delta\times \Delta\to W$ such that for all $C,D\in \Delta$, the following three conditions hold:
\begin{enumerate}  [itemindent=0cm, leftmargin=1.9cm]
\item [\textbf{(WD1)}]
$C=D$ if and only if $\delta(C,D)=1$  
\item [\textbf{(WD2)}]
if $\delta(C,D)=w$ and $D'\in \Delta$ satisfies $\delta(D,D')=s\in S$, then: 
\begin{enumerate}[label=(\roman*)]
\item
if $ws<w$, then $\delta(C,D')\in \{ws,w\}$
\item
 if $ws>w$, then $\delta(C,D')=ws$ 
\end{enumerate}
\item [\textbf{(WD3)}] 
if $\delta(C,D)=w$, then for all $s\in S$ there exists a chamber $D'\in \Delta$ such that 
\[
\delta(D,D')=s \ \ \ \  \text{and} \ \ \ \ \delta(C,D')=ws
.\]
\end{enumerate}
\end{definition}
We now show that a building is equivalently a strict $W$-groupoid on a connected simply connected groupoid. Let $(\Delta,\gd)$ be a building, and let $\cG$ be the connected simply connected groupoid with $\cG_0=\Delta$. Define a function
\[    \gd':\cG\rightarrow W,\qquad   g\mapsto \gd( \iota(g),\tau(g)      )  .\]
\noindent Then $\gd':\cG\rightarrow W$ has property (WG$1$) and is strict because $(\Delta,\gd)$ has property (WD$1$). Weakness follows from (WD$3$), as does property (WG$3$). Also, $\gd':\cG\rightarrow W$ satisfies the local triangle inequality because $(\Delta,\gd)$ satisfies (WD$2$), and so $\gd':\cG\rightarrow W$ satisfies property (WG2) by \autoref{lem:localtri}. Therefore $\gd':\cG\rightarrow W$ is a strict $W$-groupoid. Conversely, let $\gd:\cG\rightarrow W$ be a strict $W$-groupoid on a connected simply connected groupoid $\cG$. Put $\Delta=\cG_0$, and define a map
\[ \gd':\Delta\times \Delta\to W,\qquad  (C,D)\mapsto \gd(g)  \]
where $g$ is the unique edge traveling from $C$ to $D$. Then $(\Delta,\gd')$ has property (WD$1$) because $\gd:\cG\rightarrow W$ is strict and has property (WG$1$). Also, $(\Delta,\gd')$ satisfies (WD$2$) because $\gd:\cG\rightarrow W$ satisfies the local triangle inequality. Property (WD$3$) follows from weakness in the case where $ws> w$, and (WG$3$) in the case where $ws<w$. 

%%%%%%%%%%%%%%%%%%%%%%%%%%%%%%%%%%%%%%%%%%%%%%%%%%%%%%%%%%%%%%%%%%%%%%%%
%%%%%%%%%%%%%%%%%%%%%%%%%%%%%%%%%%%%%%%%%%%%%%%%%%%%%%%%%%%%%%%%%%%%%%%%
\section{$W$-Groupoids and Bruhat Decompositions} \label{bruhat} 
%%%%%%%%%%%%%%%%%%%%%%%%%%%%%%%%%%%%%%%%%%%%%%%%%%%%%%%%%%%%%%%%%%%%%%%%
%%%%%%%%%%%%%%%%%%%%%%%%%%%%%%%%%%%%%%%%%%%%%%%%%%%%%%%%%%%%%%%%%%%%%%%%

We now consider the case where $\cG$ has only one chamber. Such $W$-groupoids should correspond to chamber-transitive actions of groups on buildings. In this case, $\cG$ is naturally a group, and we put
\[G:=\cG_1\qquad  \text{and}\qquad B:=\cB_1.\] 
It follows from \autoref{prop:doublecoset} and \autoref{cor:inv} that $\delta$ factors through the map 
\[G\twoheadrightarrow  B \backslash G/ B,\qquad   g\mapsto B g B\] 
to give a function
\[\delta_{B}: B \backslash G/B\to W.\] 
Let $\delta:G\to W$ be a $W$-groupoid with one chamber such that $\delta_{B}$ is injective. Define the function
\[\cC:W\to B \backslash G/ B, \qquad \cC:=\delta^{-1}_{B}  .\] 
Then the local triangle inequality (WG2$'$) becomes
\begin{enumerate}  [itemindent=0cm, leftmargin=1.9cm]
\item [\textbf{(B$'$)}]
For all $w\in W$ and all $s\in S$, we have:
\begin{enumerate}[label=(\roman*)]
	\item
if $ws<w$, then $\cC(w)\cC(s)\subseteq \cC(ws)\cup \cC(w)$
	\item
if $ws>w$, then $\cC(w)\cC(s)\subseteq \cC(ws)$.
\end{enumerate}
\end{enumerate}  
And (WG3) becomes
\begin{enumerate}  [itemindent=0cm, leftmargin=1.9cm]
	\item [\textbf{(B$''$)}]
	For all $w\in W$ and $s\in S$ such that $ws<w$, we have
\[\cC(w)\subseteq \cC(ws)\cC(s).\]
\end{enumerate}  
If we substitute $w'=ws$, we obtain the following alternate form of (B$''$); for all $w'\in W$ and $s\in S$ such that $w's>w'$, we have
\[\cC(w's)\subseteq \cC(w')\cC(s).\]  
The following proposition shows that this inclusion holds in general.
\begin{prop} \label{without0}
Let $\delta:G\to W$ be a $W$-groupoid with one chamber such that $\delta_{B}$ is a bijection. Let $w\in W$ and $s\in S$ such that $ws<w$, then
\[\cC(ws)\subseteq \cC(w)\cC(s).\]	
\end{prop}  
\begin{proof} 
	Put $w'=ws$, then $w's>w'$ and so $\cC(w')\cC(s)\subseteq \cC(w's)$ by (B$'$). Then $\cC(ws)\cC(s)\subseteq \cC(w)$, and so $\cC(ws)\subseteq \cC(w)\cC(s)$ since $\cC(s)$ is closed under inverses by \autoref{lem:inverses2}.
\end{proof}
Then, combining (B$'$), (B$''$), and \autoref{without0}, we obtain
\begin{enumerate}  [itemindent=0cm, leftmargin=1.9cm]
	\item [\textbf{(B)}]
	For all $w\in W$ and all $s\in S$, we have:
	\begin{enumerate}[label=(\roman*)]
		\item
		if $ws<w$, then $\cC(ws)\subseteq \cC(w)\cC(s)\subseteq \cC(ws)\cup \cC(w)$
		\item
		if $ws>w$, then $\cC(w)\cC(s)= \cC(ws)$.
	\end{enumerate}
\end{enumerate}
This is the property required for a bijective function $\delta_{B}: B \backslash G/B\to W$ to be a Bruhat decomposition \cite[Section 6.2.1]{ab08}. Conversely, let $G$ be a group with a subgroup $B\leq G$, and suppose that we have a bijective function $\delta_{B}: B \backslash G/B\to W$ with property (B). Let $\delta:G\to W$ be the composition of $\delta_{B}$ with the projection $G\to B \backslash G/ B$. Then $\delta$ has properties (B$'$) and (B$''$), and so has properties (WG2$'$) and (WG3). Finally, $\delta$ is clearly weak, and has property (WG1) since if $\cC(w)=B$, then for any $s\in S$ we have
\[    \cC(ws)\subseteq \cC(s)\cC(w) =  \cC(s)   .\]
Therefore $ws=s$ and $w=1$.

%%%%%%%%%%%%%%%%%%%%%%%%%%%%%%%%%%%%%%%%%%%%%%%%%%%%%%%%%%%%%%%%%%%%%%%%
%%%%%%%%%%%%%%%%%%%%%%%%%%%%%%%%%%%%%%%%%%%%%%%%%%%%%%%%%%%%%%%%%%%%%%%%
\section{Future Work}
%%%%%%%%%%%%%%%%%%%%%%%%%%%%%%%%%%%%%%%%%%%%%%%%%%%%%%%%%%%%%%%%%%%%%%%%
%%%%%%%%%%%%%%%%%%%%%%%%%%%%%%%%%%%%%%%%%%%%%%%%%%%%%%%%%%%%%%%%%%%%%%%%

We now discus some future developments in the theory of $W$-groupoids, in particular their connection with chamber systems of type $M$, their covering theory, and their `presentations'. 

A \emph{chamber system} is an indexed collection of equivalence relations on a set, and a \emph{chamber system of type $M$} is a chamber system whose $2$-residues are buildings \cite{tits81local}. The relationship between chamber systems of type $M$ and $W$-groupoids is described in \cite{nor2}, where we show that strict $W$-groupoids whose $2$-residues are simply connected are equivalent to the chamber systems of type $M$ which are covered by buildings. A $W$-groupoid is obtained from a chamber system of type $M$ by taking the groupoid of homotopy classes of galleries. Conversely, the panel groupoids of a strict $W$-groupoid whose $2$-residues are simply connected will also be simply connected, and one recovers the associated chamber system of type $M$ by taking for equivalence relations the panel groupoids.

\medskip

In \cite{nor2}, we introduce `presentations' of strict $W$-groupoids, which we call \emph{Weyl graphs}. Roughly speaking, Weyl graphs are to strict $W$-groupoids what chambers systems are to buildings. A Weyl graph consists of an indexed collection of `generating' panel groupoids together with a collection of `relations' which we call \emph{suites}. Suites are the images of apartments of the $2$-residues of the building which covers the Weyl graph. 

Weyl graphs generalize chamber systems of type $M$ by allowing $2$-residues to be quotients of generalized polygons; they are the quotients of actions which are free on chambers, but not necessarily free on the set of $2$-residues. For example, \autoref{fig:groupoid} shows the Weyl graph of the Fano plane and of its quotient by a Singer cycle (all of the groupoids involved are equivalence relations). The galleries in the quotient which lift to apartments of the Fano plane are the suites of the quotient. We develop covering theory of strict Weyl groupoids in the language of Weyl graphs in \cite[Section 4]{nor2}, which reduces to Tits's covering theory of chamber systems of type $M$ if one assumes that coverings are injective on $2$-residues. 

Weyl graphs provide a framework in which quotients of generalized polygons by groups acting freely on chambers (flags) can be glued together to form quotients of buildings. For example, Essert's Singer lattices in \cite{essert2013geometric} of type $\wt{A}_2$ are constructed by gluing together three copies of the quotient of a projective plane by a Singer group. Using Weyl graphs, this construction is easily generalized to type $M$ in the case $m_{st}\in \{2,3,\infty\}$ for all $s,t\in S$ \cite{nor3}. 

%%%%%%%%%%%%%%%%%%%%%%%%%%%%%%%%%%%%%%%%%
\begin{figure}[t]
	\centering
	\includegraphics[scale=0.9]{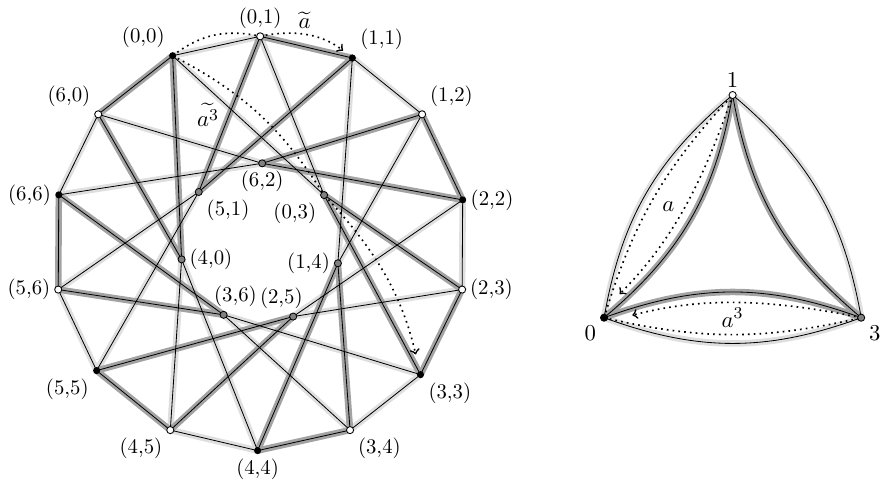}
	\caption{The Fano plane (left) and its quotient by a Singer cycle (right), constructed using the difference set $\{0,1,3\}$. All groupoids involved are equivalence relations (setoids).}
	\label{fig:groupoid}
\end{figure}
%%%%%%%%%%%%%%%%%%%%%%%%%%%%%%%%%%%%%%%%% 

\medskip

We can obtain (stacky) covering theory of $W$-groupoids directly from covering theory of groupoids. See e.g. \cite{brown06topology} for details on coverings of groupoids. The key step is to somehow allow the Borel subgroupoid to account for non-trivial isotropy. We outline the main ideas here, and leave the detailed exposition to future work. 

A \emph{covering} of $W$-groupoids $p:\cG\to \cG'$ is a surjective groupoid homomorphism which preserves $W$-length such that for all chambers $C\in \cG$, the restriction of $p$ to the edges which issue from $C$ is a bijection into the edges which issue from $p(C)$. Naturally associated to the \emph{free} action of a group $G$ on a $W$-groupoid $\cG$ is the quotient $\cG\to G \backslash \cG$, which is a covering map of $W$-groupoids. Conversely, the fundamental group of a \emph{strict} $W$-groupoid $\cG$ acts freely on the building which covers $\cG$. 

If the action of a group on a building is not free, one can move to a free action by replacing each chamber $C$ by $\text{stab}(C)$-many chambers, all at distance $1\in W$. One then takes the ordinary quotient, which will have a non-trivial Borel subgroupoid which is a bundle of groups. Conversely, the fundamental group of a (possibly non-strict) $W$-groupoid acts freely on its universal cover, which develops non-trivial isotropy when one moves to a building by identifying chambers at distance $1\in W$. 

One should be able to define `presentations' of $W$-groupoids in general, which might be called \emph{Weyl graphs of groups}, and develop covering theory in this language. Tits's amalgams are Weyl graphs of groups with a single chamber since the associated $W$-groupoid is generated from a collection of panel groups, whose suites are determined by the amalgamation data.

\bibliographystyle{alpha}
\bibliography{sample}

%\nocite{*}
%\printindex

\end{document}